\documentclass[11pt]{article}
\usepackage[margin=1in]{geometry}
\usepackage{amsmath,amssymb,amsthm}
\usepackage{enumitem}
\usepackage{array}
\usepackage[colorlinks=true,linkcolor=blue,citecolor=blue,urlcolor=blue]{hyperref}

\newtheorem{theorem}{Theorem}[section]
\newtheorem{lemma}[theorem]{Lemma}
\newtheorem{proposition}[theorem]{Proposition}
\newtheorem{corollary}[theorem]{Corollary}
\newtheorem{conjecture}[theorem]{Conjecture}
\theoremstyle{definition}
\newtheorem{definition}[theorem]{Definition}
\newtheorem{example}[theorem]{Example}
\newtheorem{remark}[theorem]{Remark}

\newcommand{\Z}{\mathbb{Z}}
\newcommand{\id}{\mathrm{id}}

\title{Symmetric harmoniousness of odd-order groups}
\author{Mohammad Javaheri\\[2pt]
\small Department of Mathematics and Statistics, Siena University\\
\small 515 Loudon Road, Loudonville, NY 12211, USA\\
\small \texttt{mjavaheri@siena.edu}}
\date{}

\begin{document}
\maketitle

\begin{abstract}
We prove that every group of odd order is symmetric harmonious, i.e., there
exists a permutation $g_0,g_1,\dots,g_{\ell-1}$ of the elements of $G$ such that
the consecutive products $g_0g_1,g_1g_2,\dots,g_{\ell-1}g_0$ also form a
permutation of the elements of $G$ and $g_{\ell-i}=g_i^{-1}$ for all
$1\le i\le \ell-1$, where $\ell=|G|$. Consequently, a group is symmetric
harmonious if and only if it has odd order. We also characterize symmetric
harmonious sequences in terms of complete mappings. As applications we obtain
new examples of R*-sequenceable groups, and we show that every split extension
of an odd-order group by a harmoniously matched group is both harmonious and
\#-harmonious. Together with several new harmoniously matched groups
($D_{12},D_{16},D_{20},D_{24},Q_{16},Q_{24},A_4,S_4$), this yields many
non-abelian groups that are both harmonious and \#-harmonious, including
infinite families of dihedral and dicyclic groups and all $2$-nilpotent groups
with an abelian Sylow $2$-subgroup that is neither cyclic nor elementary abelian
and is central in its normalizer.
\end{abstract}


\section{Introduction}

A \emph{harmonious sequence} in a finite group $G$ is a permutation
$g_0,g_1,\dots,g_{\ell-1}$ of the elements of $G$ such that the consecutive
products $g_0g_1,g_1g_2,\dots,g_{\ell-1}g_0$ also form a permutation of the
elements of $G$ (where $\ell=|G|$). A group is \emph{harmonious} if it admits a
harmonious sequence. Such sequences were introduced by Beals et
al.~\cite{BGHJ} in 1991 in connection with the study of complete mappings. A
\emph{complete mapping} of $G$ is a bijection $\varphi\colon G\to G$ such that
the map $x\mapsto x\varphi(x)$ is also a bijection.

If $g_0,g_1,\dots,g_{\ell-1}$ is a harmonious sequence in $G$, then the map
\[
\varphi\colon g_i\mapsto g_{i+1}\qquad (0\le i\le \ell-1,\ g_\ell=g_0)
\]
is a complete mapping of $G$ that is a single cycle of length $|G|$.
Conversely, if $\varphi$ is a complete mapping of $G$ that is a single
$|G|$-cycle, then for any $g\in G$ the sequence
$g,\varphi(g),\varphi^2(g),\dots,\varphi^{\ell-1}(g)$ is a harmonious sequence
in $G$.

Hall and Paige~\cite{HP} proved that a necessary condition for a finite group
$G$ to admit a complete mapping is that the Sylow $2$-subgroups of $G$ are
trivial or non-cyclic (the \emph{Hall--Paige condition}). Their conjecture that
this condition is also sufficient was proved by the combined work of
Wilcox~\cite{Wil}, who reduced it to the sporadic simple groups,
Evans~\cite{Eva}, who handled all sporadic groups except $J_4$, and Bray, whose
treatment of $J_4$ appears in~\cite{BCCSZ}.

Beals et al.~\cite{BGHJ} showed that every abelian group is harmonious except
for the non-trivial elementary abelian $2$-groups and the binary groups (groups
with a unique element of order $2$); see also~\cite{DJ} for harmonious
sequences in non-abelian groups with a unique involution. Among non-abelian
groups, the dihedral and dicyclic groups of order $8n$ with $n>1$ are
harmonious~\cite{BGHJ,Wan,WL00}. More recently, M\"uyesser and
Pokrovskiy~\cite{MP} proved that every sufficiently large group satisfying the
Hall--Paige condition, other than an elementary abelian $2$-group, is
harmonious.

Wang and Leonard~\cite{WL94} introduced the notion of symmetric harmonious
groups and used such groups to construct new examples of R*-sequenceable
groups.

\begin{definition}\label{def:main}
Let $G$ be a finite group of order $\ell$.
\begin{enumerate}[label=(\roman*)]
\item $G$ is \emph{symmetric harmonious} if there exists a harmonious sequence
$g_0,g_1,\dots,g_{\ell-1}$ in $G$ such that
\[
g_ig_{\ell-i}=1_G\qquad\text{for all }1\le i\le \ell-1 .
\]
Such a sequence is called a \emph{symmetric harmonious sequence}.
\item $G$ is \emph{R-sequenceable} if there exists a permutation
$g_1,\dots,g_{\ell-1}$ of the non-identity elements of $G$ such that the
consecutive quotients
$g_1^{-1}g_2,\ g_2^{-1}g_3,\ \dots,\ g_{\ell-1}^{-1}g_1$
also form a permutation of the non-identity elements of $G$. If, in addition,
there is an index $i$ such that $g_{i-1}g_{i+1}=g_{i+1}g_{i-1}=g_i$, then $G$ is
\emph{R*-sequenceable}.
\end{enumerate}
\end{definition}

Wang and Leonard~\cite{WL94} proved that if $q=6h+1$ is prime, then the
non-abelian group of order $3q$ is symmetric harmonious, and they showed
in~\cite{WL95} that every nilpotent group of odd order is symmetric harmonious.
The main result of this article settles the question completely.

\begin{theorem}\label{thm:main}
A finite group is symmetric harmonious if and only if it has odd order.
\end{theorem}

The proof (Section~\ref{sec:sym}) is self-contained apart from the
Feit--Thompson theorem: we lift symmetric harmonious sequences through normal
extensions (Lemma~\ref{lem:lift}), starting from the trivial sequence
$0,1,\dots,p-1$ in $\Z_p$. In particular, the abelian case of~\cite{BGHJ,WL00}
and the nilpotent case of~\cite{WL95} are recovered. We also show
(Proposition~\ref{prop:cm}) that symmetric harmonious sequences correspond
exactly to $|G|$-cycle complete mappings $\varphi$ satisfying
$\iota\varphi\iota=\varphi^{-1}$, where $\iota(x)=x^{-1}$.

The Hall--Paige condition is also necessary for the existence of
R-sequencings, since an R-sequencing yields a complete mapping~\cite{Pai}.
Friedlander et al.~\cite{FGM} introduced R*-sequenceable groups and used them to
build R-sequencings of abelian groups from R*-sequencings of their Sylow
$2$-subgroups. Many abelian groups are known to be R*-sequenceable, including
classes of abelian groups with a non-cyclic Sylow $2$-subgroup~\cite{FGM,Hea}
and odd-order abelian groups with restricted Sylow $3$-subgroup~\cite{FGM,OW}.
Non-abelian examples include the non-abelian groups of order $pq$ with $p,q$
distinct odd primes~\cite{WL00}, the dihedral groups of order $4n$ with $n\ge 4$,
and certain dicyclic groups~\cite{WL94}.

Wang and Leonard~\cite[Prop.~1]{WL94} proved that if $G$ is an R*-sequenceable
group of even order and $H$ is symmetric harmonious, then $G\times H$ is
R*-sequenceable. Combining this with Theorem~\ref{thm:main} gives the following.

\begin{theorem}\label{thm:Rstar}
If $G$ is an R*-sequenceable group of even order and $H$ is any group of odd
order, then $G\times H$ is R*-sequenceable.
\end{theorem}

Beals et al.~\cite{BGHJ} also studied the variant of harmoniousness in which
the identity is excluded.

\begin{definition}\label{def:sharp}
A group $G$ of order $\ell\ge 2$ is \emph{\#-harmonious} if there exists a
permutation $g_1,g_2,\dots,g_{\ell-1}$ of the non-identity elements of $G$ such
that the consecutive products $g_1g_2,g_2g_3,\dots,g_{\ell-1}g_1$ also form a
permutation of the non-identity elements of $G$.
\end{definition}

Setting $\varphi(g_i)=g_{i+1}$ and $\varphi(1_G)=1_G$ shows that a
\#-harmonious group admits a complete mapping, so \#-harmonious groups satisfy
the Hall--Paige condition. However, unlike the harmonious case, all non-cyclic
elementary abelian $2$-groups are \#-harmonious~\cite{BGHJ}. Moreover, every
abelian group that is not binary is \#-harmonious, except for
$\Z_3$~\cite{BGHJ}. The following theorem gives, to the author's knowledge, the
first examples of non-abelian \#-harmonious groups. Harmoniously matched groups
are defined in Definition~\ref{def:hm}; we use the convention that in
$K\ltimes H$ the subgroup $H$ is normal.

\begin{theorem}\label{thm:sharp}
Let $K$ be a harmoniously matched group and let $H_1,\dots,H_k$ be groups of
odd order. Then every iterated semidirect product
\[
G=\bigl(\cdots\bigl((K\ltimes H_1)\ltimes H_2\bigr)\ltimes\cdots\bigr)\ltimes H_k
\]
is harmoniously matched; in particular, $G$ is both harmonious and
\#-harmonious.
\end{theorem}

By~\cite{BGHJ}, every abelian group that is neither binary, nor an elementary
abelian $2$-group, nor isomorphic to $\Z_3$ (nor trivial) is harmoniously
matched. In Proposition~\ref{prop:small} we show that the non-abelian groups
\[
D_{12},\ D_{16},\ D_{20},\ D_{24},\ Q_{16},\ Q_{24},\ A_4,\ S_4
\]
are also harmoniously matched (here $D_{2n}$ is the dihedral group of order $2n$
and $Q_{4n}$ is the dicyclic group of order $4n$). Combining these with
Theorem~\ref{thm:sharp}, the Schur--Zassenhaus theorem and Burnside's normal
$p$-complement theorem gives the following.

\begin{corollary}\label{cor:families}
Each of the following groups is both harmonious and \#-harmonious.
\begin{enumerate}[label=(\alph*)]
\item Any group $G$ with a normal subgroup $N$ of odd order such that
$\gcd(|N|,|G/N|)=1$ and $G/N$ is harmoniously matched; for instance, $G/N$
abelian and neither trivial, binary, elementary abelian of exponent $2$, nor
isomorphic to $\Z_3$.
\item Any group $G$ whose Sylow $2$-subgroup $S$ is abelian, neither cyclic nor
elementary abelian, and satisfies $N_G(S)=C_G(S)$.
\item The dihedral groups $D_{16m}$ and dicyclic groups $Q_{16m}$ for odd $m$,
the groups $D_{12m}$, $D_{24m}$, $Q_{24m}$ for $\gcd(m,6)=1$, and $D_{20m}$ for
$\gcd(m,10)=1$.
\item $K\ltimes H$ for any odd-order group $H$ and any action of
$K\in\{A_4,S_4,D_{12},D_{16},D_{20},D_{24},Q_{16},Q_{24}\}$ on $H$; e.g.\
$A_4\times H$ and $S_4\times H$.
\end{enumerate}
\end{corollary}

Finally, for groups of odd order we obtain partial results towards the
following conjecture, which holds for all groups of odd order at most $27$
(by~\cite{BGHJ} for abelian groups and by computer search for the non-abelian
ones) and for several larger groups (Section~\ref{sec:odd}).

\begin{conjecture}\label{conj:odd}
Every group of odd order, other than the trivial group and $\Z_3$, is
\#-harmonious.
\end{conjecture}

\section{Symmetric harmonious sequences}\label{sec:sym}

We begin with some elementary properties of symmetric harmonious sequences.

\begin{lemma}\label{lem:basic}
Let $g_0,\dots,g_{\ell-1}$ be a symmetric harmonious sequence in a group $G$ of
order $\ell$, and extend the indices modulo $\ell$.
\begin{enumerate}[label=(\alph*)]
\item $\ell$ is odd and $g_0=1_G$; hence $g_{-i}=g_i^{-1}$ for all $i\in\Z$.
\item The consecutive products $c_i=g_ig_{i+1}$ satisfy $c_{\ell-1-i}=c_i^{-1}$
for all $i$. In particular, $c_i=1_G$ if and only if $i\equiv(\ell-1)/2
\pmod{\ell}$.
\end{enumerate}
\end{lemma}

\begin{proof}
(a) If $\ell=1$ there is nothing to prove, so let $\ell\ge 2$. Since the
products $g_ig_{i+1}$ run through all of $G$, there is an index $i$ with
$g_{i-1}g_i=1_G$ (indices modulo $\ell$). If $i=0$, then
$g_{\ell-1}g_0=1_G$, so $g_0=g_{\ell-1}^{-1}=g_1$, contradicting $\ell\ge 2$.
Hence $1\le i\le\ell-1$ and $g_{i-1}=g_i^{-1}=g_{\ell-i}$. Since both
$i-1$ and $\ell-i$ lie in $\{0,\dots,\ell-1\}$, we get $i-1=\ell-i$, so
$\ell=2i-1$ is odd. Now the elements $g_1,\dots,g_{\ell-1}$ are partitioned into
the pairs $\{g_i,g_{\ell-i}\}$ with $i\neq\ell-i$; since $g_{\ell-i}=g_i^{-1}$
and $g_i\ne g_{\ell-i}$, no $g_i$ with $i\ne 0$ is equal to its own inverse. In
particular, $g_i\ne 1_G$ for $i\neq 0$, so $g_0=1_G$. Then $g_{-i}=g_i^{-1}$
holds for every $i$.

(b) By (a), $c_{\ell-1-i}=g_{-1-i}g_{-i}=g_{i+1}^{-1}g_i^{-1}=c_i^{-1}$. Since
the $c_i$ are distinct and $\ell$ is odd, $c_i=1_G$ forces $c_i=c_i^{-1}=
c_{\ell-1-i}$, i.e.\ $2i\equiv \ell-1\pmod\ell$, i.e.\ $i\equiv(\ell-1)/2$.
\end{proof}

Let $\iota\colon G\to G$ denote inversion, $\iota(x)=x^{-1}$.

\begin{proposition}\label{prop:cm}
A finite group $G$ is symmetric harmonious if and only if it has a complete
mapping $\varphi$ that is a single $|G|$-cycle and satisfies
$\iota\varphi\iota=\varphi^{-1}$ (equivalently, $\iota\varphi$ is an
involution). More precisely, if $\varphi$ is such a complete mapping, then
$g_i=\varphi^i(1_G)$ $(0\le i\le|G|-1)$ is a symmetric harmonious sequence, and
every symmetric harmonious sequence arises this way.
\end{proposition}

\begin{proof}
Let $\ell=|G|$. Suppose $g_0,\dots,g_{\ell-1}$ is symmetric harmonious, and let
$\varphi(g_i)=g_{i+1}$ (indices modulo $\ell$). Then $\varphi$ is a complete
mapping and an $\ell$-cycle. By Lemma~\ref{lem:basic}(a), $\iota(g_j)=g_{-j}$
for every $j$, so
$\iota\varphi\iota(g_i)=\iota\varphi(g_{-i})=\iota(g_{1-i})=g_{i-1}=
\varphi^{-1}(g_i)$. Also $g_i=\varphi^i(g_0)=\varphi^i(1_G)$.

Conversely, let $\varphi$ be such a complete mapping and set
$g_i=\varphi^i(1_G)$. Since $\varphi$ is an $\ell$-cycle, $g_0,\dots,g_{\ell-1}$
is a harmonious sequence. From $\iota\varphi\iota=\varphi^{-1}$ we get
$\iota\varphi^i\iota=\varphi^{-i}$, hence
$g_i^{-1}=\iota\varphi^i(1_G)=\iota\varphi^i\iota(1_G)=\varphi^{-i}(1_G)=
g_{\ell-i}$ for $1\le i\le\ell-1$.
\end{proof}

\begin{example}\label{ex:cyclic}
If $n$ is odd, then $0,1,2,\dots,n-1$ is a symmetric harmonious sequence in
$\Z_n$: the consecutive sums are $2i+1$ $(0\le i\le n-1)$, which form a
permutation of $\Z_n$ because $2$ is invertible modulo $n$, and
$i+(n-i)\equiv 0$. (The corresponding complete mapping is $x\mapsto x+1$.)
\end{example}

Since odd-order groups are solvable, it suffices to describe how symmetric
harmonious sequences can be lifted through normal extensions. We first
illustrate the construction with an example. Let $p,q$ be odd primes with
$q\mid p-1$, and consider the semidirect product $G=\Z_p\rtimes\Z_q$,
\[
G=\langle a,b\mid a^p=b^q=1,\ bab^{-1}=a^r\rangle,
\]
where $r$ has multiplicative order $q$ modulo $p$. Write
$t_j=(j-1)(j-2)/2$. Then the sequence whose $i$-th row $(1\le i\le p)$ is
\[
a^{2i-2},\quad b^{-t_2}a^{2i-1}b^{t_3},\quad b^{-t_3}a^{2i-1}b^{t_4},\quad
\dots,\quad b^{-t_q}a^{2i-1}b^{t_{q+1}},
\]
read row by row, is a symmetric harmonious sequence in $G$. Explicitly, the
first rows are
\[
\begin{array}{llllll}
1, & ab, & b^{-1}ab^{3}, & b^{-3}ab^{6}, & \dots, & b^{-(q-2)(q-1)/2}a,\\
a^{2}, & a^{3}b, & b^{-1}a^{3}b^{3}, & b^{-3}a^{3}b^{6}, & \dots, & b^{-(q-2)(q-1)/2}a^{3},
\end{array}
\]
(note that $t_{q+1}=q(q-1)/2\equiv 0\pmod q$). For $p=7$, $q=3$, $r=2$, this
sequence is
\[
1,\ ab,\ a^4b^2,\ a^2,\ a^3b,\ a^5b^2,\ a^4,\ a^5b,\ a^6b^2,\ a^6,\ b,\ b^2,\
a,\ a^2b,\ ab^2,\ a^3,\ a^4b,\ a^2b^2,\ a^5,\ a^6b,\ a^3b^2,
\]
with consecutive products
\[
ab,\ a^2,\ a^5b^2,\ a^5b,\ a^6,\ b^2,\ a^2b,\ a^3,\ a^2b^2,\ a^6b,\ 1,\
a^4b^2,\ a^3b,\ a^4,\ a^6b^2,\ b,\ a,\ ab^2,\ a^4b,\ a^5,\ a^3b^2 .
\]
This is the special case $H=\langle a\rangle$, $h_k=a^k$, $k_r=b^r$ of the
following lemma (one checks that $k_r\mu_r=b^{-t_{r+1}}$ and
$\mu_r^{-1}=b^{t_{r+2}}$ in the notation of the proof).

\begin{lemma}\label{lem:lift}
Let $H$ be a normal subgroup of a finite group $G$. If $H$ and $G/H$ are
symmetric harmonious, then so is $G$.
\end{lemma}

\begin{proof}
Let $n=|H|$ and $m=|G/H|$; both are odd by Lemma~\ref{lem:basic}(a). If $m=1$
there is nothing to prove, so assume $m\ge 3$. Let $K_0,K_1,\dots,K_{m-1}$ be a
symmetric harmonious sequence in $G/H$; by Lemma~\ref{lem:basic}(a), $K_0=H$.
Put $t=(m-1)/2$. Choose $k_0=1_G$, choose $k_i\in K_i$ arbitrarily for
$1\le i\le t$, and set $k_{m-i}=k_i^{-1}$ for $1\le i\le t$; then $k_{m-i}\in
K_i^{-1}=K_{m-i}$. Thus
\begin{equation}\label{eq:k}
k_i\in K_i\quad(0\le i\le m-1),\qquad k_ik_{m-i}=1_G\quad(1\le i\le m-1).
\end{equation}
For $0\le i\le m-1$ let $\mu_i$ be the product of the terms
$k_{i+1},\dots,k_{m-1}$:
\[
\mu_i=k_{i+1}k_{i+2}\cdots k_{m-1}\quad(0\le i\le m-2),\qquad \mu_{m-1}=1_G .
\]
Note that $\mu_{i-1}=k_i\mu_i$ for $1\le i\le m-1$. By~\eqref{eq:k},
\begin{equation}\label{eq:mu0}
\mu_0=k_1\cdots k_t\,k_{t+1}\cdots k_{m-1}
=k_1\cdots k_t\,k_t^{-1}\cdots k_1^{-1}=1_G .
\end{equation}
Consequently, for $1\le r\le m-1$,
\begin{equation}\label{eq:sym}
k_{m-r}\mu_{m-r}=k_{m-r}k_{m-r+1}\cdots k_{m-1}
=k_r^{-1}k_{r-1}^{-1}\cdots k_1^{-1}=(k_1\cdots k_r)^{-1}
=k_{r+1}\cdots k_{m-1}=\mu_r ,
\end{equation}
where the penultimate equality uses $\mu_0=1_G$ (for $r=m-1$ both sides equal
$1_G$).

Let $h_0,h_1,\dots,h_{n-1}$ be a symmetric harmonious sequence in $H$, with
indices taken modulo $n$; thus $h_0=1_G$ and $h_jh_{-j}=1_G$ for all $j$. For
$0\le p\le n-1$ and $0\le r\le m-1$ define
\begin{equation}\label{eq:g}
g_{pm+r}=\begin{cases}
h_{2p} & \text{if } r=0,\\
k_r\mu_r\,h_{2p+1}\,\mu_r^{-1} & \text{if } 1\le r\le m-1 .
\end{cases}
\end{equation}
Since $H$ is normal in $G$, we have $g_{pm+r}\in k_rH=K_r$. We claim that
$g_0,\dots,g_{mn-1}$ is a symmetric harmonious sequence in $G$. Indices of the
$g$'s are taken modulo $mn$.

\smallskip
\noindent\emph{Step 1: $g_0,\dots,g_{mn-1}$ is a permutation of $G$.} Let
$i=pm+r$ and $j=qm+s$ with $0\le p,q\le n-1$, $0\le r,s\le m-1$, and suppose
$g_i=g_j$. Since $g_i\in K_r$ and $g_j\in K_s$, we get $r=s$. If $r=0$, then
$h_{2p}=h_{2q}$, so $2p\equiv 2q\pmod n$, and $p=q$ because $n$ is odd. If
$r>0$, then cancelling $k_r\mu_r$ and $\mu_r^{-1}$ gives $h_{2p+1}=h_{2q+1}$,
and again $p=q$. Thus $i=j$.

\smallskip
\noindent\emph{Step 2: the products $g_ig_{i+1}$ are distinct.} Let $i,j$ be
as in Step~1 and suppose $g_ig_{i+1}=g_jg_{j+1}$. Since
$g_ig_{i+1}\in K_rK_{r+1}$ (with $K_m=K_0$) and the cosets $K_rK_{r+1}$,
$0\le r\le m-1$, are distinct, we get $r=s$. We show $p=q$.

\emph{Case 1: $1\le r\le m-2$.} Then $i+1=pm+(r+1)$ and, using
$\mu_r^{-1}k_{r+1}\mu_{r+1}=1_G$ and $k_r\mu_r=\mu_{r-1}$,
\[
g_ig_{i+1}=k_r\mu_rh_{2p+1}\mu_r^{-1}\cdot k_{r+1}\mu_{r+1}h_{2p+1}\mu_{r+1}^{-1}
=\mu_{r-1}\,h_{2p+1}^2\,\mu_{r+1}^{-1}.
\]
Hence $h_{2p+1}^2=h_{2q+1}^2$. In a group of odd order the map $h\mapsto h^2$ is
a bijection, so $h_{2p+1}=h_{2q+1}$ and $p=q$.

\emph{Case 2: $r=0$.} Then $g_{i+1}=k_1\mu_1h_{2p+1}\mu_1^{-1}=
\mu_0h_{2p+1}\mu_1^{-1}=h_{2p+1}\mu_1^{-1}$, so
$g_ig_{i+1}=(h_{2p}h_{2p+1})\mu_1^{-1}$. Hence $h_{2p}h_{2p+1}=h_{2q}h_{2q+1}$.
Since $h_0,\dots,h_{n-1}$ is harmonious, $2p\equiv 2q\pmod n$, so $p=q$.

\emph{Case 3: $r=m-1$.} Then $i+1=(p+1)m$, so $g_{i+1}=h_{2p+2}$ (this also holds
for $p=n-1$, since $g_{mn}=g_0=h_0=h_{2n}$). As $\mu_{m-1}=1_G$,
\[
g_ig_{i+1}=k_{m-1}\,h_{2p+1}h_{2p+2}.
\]
Hence $h_{2p+1}h_{2p+2}=h_{2q+1}h_{2q+2}$, so $2p+1\equiv 2q+1\pmod n$ and $p=q$.

Since there are $mn$ products and they are distinct, they form a permutation
of $G$.

\smallskip
\noindent\emph{Step 3: $g_ig_{mn-i}=1_G$ for $1\le i\le mn-1$.} Write $i=pm+r$
as above. If $r>0$, then $mn-i=m(n-1-p)+(m-r)$ with $1\le m-r\le m-1$, so by
\eqref{eq:sym}, $\mu_r^{-1}k_{m-r}\mu_{m-r}=1_G$ and
\[
g_ig_{mn-i}=k_r\mu_rh_{2p+1}\mu_r^{-1}\cdot k_{m-r}\mu_{m-r}h_{2n-2p-1}
\mu_{m-r}^{-1}
=k_r\mu_r\,h_{2p+1}h_{-(2p+1)}\,\mu_{m-r}^{-1}=k_r\mu_r\mu_{m-r}^{-1}.
\]
(Here $h_{2p+1}h_{-(2p+1)}=1_G$ also when $2p+1\equiv 0\pmod n$, since then
both factors equal $h_0=1_G$.) By~\eqref{eq:sym} again,
$k_r\mu_r\mu_{m-r}^{-1}=k_r(k_{m-r}\mu_{m-r})\mu_{m-r}^{-1}=k_rk_{m-r}=1_G$.
If $r=0$, then $1\le p\le n-1$, $mn-i=(n-p)m$, and
$g_ig_{mn-i}=h_{2p}h_{2n-2p}=h_{2p}h_{-2p}=1_G$.

Thus $G$ is symmetric harmonious.
\end{proof}

\begin{proof}[Proof of Theorem~\ref{thm:main}]
If $G$ is symmetric harmonious, then $|G|$ is odd by Lemma~\ref{lem:basic}(a).

Conversely, we prove by induction on $|G|$ that every group $G$ of odd order is
symmetric harmonious. The trivial group is symmetric harmonious. Let
$|G|>1$. By the Feit--Thompson theorem~\cite{FT}, $G$ is solvable. If $G$ has
prime order, then $G\cong\Z_p$ is symmetric harmonious by
Example~\ref{ex:cyclic}. Otherwise, since a solvable group whose only normal
subgroups are $1$ and itself has prime order, $G$ has a normal subgroup $H$
with $1<|H|<|G|$. Both $H$ and $G/H$ have odd order smaller than $|G|$, hence are
symmetric harmonious by induction, and $G$ is symmetric harmonious by
Lemma~\ref{lem:lift}.
\end{proof}

\begin{remark}
The argument does not use Feit--Thompson for solvable groups: Lemma~\ref{lem:lift}
and Example~\ref{ex:cyclic} alone show that every solvable group of odd order is
symmetric harmonious. Moreover, the proof is constructive: a symmetric
harmonious sequence can be computed recursively along any chief (or composition)
series. We implemented this recursion and checked the output for a range of
odd-order groups, including several non-nilpotent groups of orders up to $171$.
\end{remark}

\begin{corollary}\label{cor:oddharm}
Every group $G$ of odd order is harmonious; in fact $G$ has a complete mapping
$\varphi$ that is a single $|G|$-cycle and satisfies
$\varphi(\varphi(x)^{-1})=x^{-1}$ for all $x\in G$.
\end{corollary}

\begin{proof}
Combine Theorem~\ref{thm:main} with Proposition~\ref{prop:cm}; the identity
$\iota\varphi\iota\varphi=\id$ reads $\varphi(\varphi(x)^{-1})^{-1}=x$.
\end{proof}

\section{\texorpdfstring{\#}{\#}-harmonious groups}\label{sec:sharp}

In this section we construct \#-harmonious groups as semidirect products of
odd-order groups by harmoniously matched groups.

\begin{definition}\label{def:hm}
A finite group $G$ of order $n\ge 2$ is \emph{harmoniously matched} if there
exist a harmonious sequence $g_0,g_1,\dots,g_{n-1}$ and a \#-harmonious
sequence $g'_1,g'_2,\dots,g'_{n-1}$ in $G$ such that
\begin{equation}\label{eq:hm}
g_0=g'_1,\qquad g_{n-1}=g'_{n-1},\qquad g_0g_1\cdots g_{n-1}=1_G .
\end{equation}
\end{definition}

\begin{remark}
One could also require $g'_1g'_2\cdots g'_{n-1}=1_G$, but this condition is
never needed below, so we omit it. Beals et al.~\cite{BGHJ} showed that every
non-trivial finite abelian group other than $\Z_3$, the binary groups and the
elementary abelian $2$-groups admits sequences satisfying the first two
conditions in~\eqref{eq:hm}. For abelian groups the third condition is
automatic: the product of all elements of a finite abelian group is $1_G$
unless the group is binary. Note also that the endpoints $g_0=g'_1$ and
$g_{n-1}=g'_{n-1}$ are necessarily non-identity elements.
\end{remark}

The following lemma is the key construction. It requires only that the
product of the harmonious sequence centralizes the normal subgroup.

\begin{lemma}\label{lem:split}
Let $G=HK$ with $H\trianglelefteq G$, $H\cap K=\{1_G\}$, and $|H|=m$ odd.
Suppose $K$ (of order $n\ge 2$) has a harmonious sequence
$u_0,\dots,u_{n-1}$ and a \#-harmonious sequence $u'_1,\dots,u'_{n-1}$ with
$u'_1=u_0$ and $u'_{n-1}=u_{n-1}$, and that $z:=u_0u_1\cdots u_{n-1}$
centralizes $H$. Then $G$ has a harmonious sequence $S$ and a \#-harmonious
sequence $S^\#$ with the same first terms and the same last terms, and the
product of the terms of $S$ equals $z^m$. In particular, if $K$ is harmoniously
matched, then so is $G$.
\end{lemma}

\begin{proof}
By Theorem~\ref{thm:main}, $H$ has a symmetric harmonious sequence
$h_0,h_1,\dots,h_{m-1}$; by Lemma~\ref{lem:basic}(a), $h_0=1_G$, and we take
indices modulo $m$. Put $\sigma_{-1}=1_G$ and $\sigma_k=u_0u_1\cdots u_k$
$(0\le k\le n-1)$, so $\sigma_{n-1}=z$. For $1\le j\le m-1$ and $0\le k\le n-1$
let
\[
x_{j,k}=\sigma_{k-1}^{-1}\,h_j\,\sigma_k ,
\]
and let $S$ (respectively $S^\#$) be the sequence
\[
\underbrace{u_0,\dots,u_{n-1}}_{\text{resp. } u'_1,\dots,u'_{n-1}},\
x_{1,0},x_{1,1},\dots,x_{1,n-1},\ x_{2,0},\dots,x_{2,n-1},\ \dots,\
x_{m-1,0},\dots,x_{m-1,n-1}.
\]
We call the first $n$ (resp.\ $n-1$) terms the \emph{first block} and the terms
$x_{j,0},\dots,x_{j,n-1}$ the \emph{$j$-th row}. For an element $y\in G$ let
$c_y$ be the automorphism $h\mapsto y^{-1}hy$ of $H$.

\emph{The terms.} Since $\sigma_k=\sigma_{k-1}u_k$, we have
$x_{j,k}=c_{\sigma_{k-1}}(h_j)\,u_k$. As $j$ runs over $1,\dots,m-1$, $h_j$ runs
over $H\setminus\{1_G\}$, so the $k$-th entries of the rows run over
$(H\setminus\{1_G\})u_k$. Since $G$ is the disjoint union of the cosets $Hy$,
$y\in K$, the terms of $S$ form exactly $K\cup\bigcup_k(H\setminus\{1_G\})u_k=G$,
and the terms of $S^\#$ form $(K\setminus\{1_G\})\cup\bigcup_k
(H\setminus\{1_G\})u_k=G\setminus\{1_G\}$, each element occurring once.

\emph{The products.} There are three kinds of consecutive products.
\begin{enumerate}[label=(\roman*)]
\item Inside the first block: $u_ku_{k+1}$ $(0\le k\le n-2)$ for $S$, and
$u'_lu'_{l+1}$ $(1\le l\le n-2)$ for $S^\#$.
\item Inside the $j$-th row, for $0\le k\le n-2$:
\[
x_{j,k}x_{j,k+1}=\sigma_{k-1}^{-1}h_j^2\sigma_{k+1}
=c_{\sigma_{k-1}}(h_j^2)\,u_ku_{k+1}.
\]
\item Junctions. Since $z$ centralizes $H$ and $\sigma_{n-2}^{-1}z=u_{n-1}$,
the last entry of every row is $x_{j,n-1}=\sigma_{n-2}^{-1}h_jz=u_{n-1}h_j$,
while the first entry is $x_{j,0}=h_ju_0$. Let $y=u_{n-1}u_0$ and
$d=c_{u_{n-1}^{-1}}$, i.e.\ $d(h)=u_{n-1}hu_{n-1}^{-1}$. The junction from row
$j$ to row $j+1$ gives $u_{n-1}h_jh_{j+1}u_0=d(h_jh_{j+1})\,y$; the junction from
the first block to row $1$ gives $u_{n-1}\cdot h_1u_0=d(h_0h_1)\,y$ (recall
$u'_{n-1}=u_{n-1}$ and $h_0=1_G$); and the cyclic junction from row $m-1$ back
to the first block gives $u_{n-1}h_{m-1}\cdot u_0=d(h_{m-1}h_0)\,y$ (recall
$u'_1=u_0$). If $m=1$ there are no rows and the single junction product is
$u_{n-1}u_0=y$.
\end{enumerate}
Because $h_0,\dots,h_{m-1}$ is harmonious, the junction products are
$d(\{h_0h_1,h_1h_2,\dots,h_{m-1}h_0\})\,y=d(H)y=Hy$, each exactly once.

Since $u_0,\dots,u_{n-1}$ is harmonious in $K$, the elements $u_ku_{k+1}$
$(0\le k\le n-2)$ and $y$ are the $n$ distinct elements of $K$, so the cosets
$Hu_ku_{k+1}$ $(0\le k\le n-2)$ and $Hy$ are the $n$ distinct cosets of $H$.
Fix $k\le n-2$ and put $w=u_ku_{k+1}$. Since $m$ is odd, $h\mapsto h^2$ permutes
$H\setminus\{1_G\}$, so the products of type (ii) lying in $Hw$ are exactly
$c_{\sigma_{k-1}}(H\setminus\{1_G\})w=(H\setminus\{1_G\})w$. For $S$, the
remaining element $w$ of $Hw$ is produced once by type (i), and the coset $Hy$
is covered by the junctions; type (i) produces nothing in $Hy$. Hence the
products of $S$ form a permutation of $G$.

For $S^\#$, the cyclic products of the \#-harmonious sequence
$u'_1,\dots,u'_{n-1}$ form $K\setminus\{1_G\}$, and the cyclic product
$u'_{n-1}u'_1$ equals $y$. Hence $y\ne 1_G$, and the products of type (i)
form $K\setminus\{1_G,y\}$; that is, they supply the element $w$ of each coset
$Hw$ $(w=u_ku_{k+1},\ k\le n-2)$ exactly when $w\ne 1_G$. Together with types
(ii) and (iii), the products of $S^\#$ form a permutation of
$G\setminus\{1_G\}$.

\emph{Endpoints and product.} Both sequences begin with $u_0=u'_1$ and end with
$x_{m-1,n-1}$ (or with $u_{n-1}=u'_{n-1}$ if $m=1$). The product of the $j$-th
row telescopes to $h_j^n\sigma_{n-1}=h_j^nz$, so, as $z$ centralizes $H$,
\[
\prod S=z\prod_{j=1}^{m-1}(h_j^nz)=z^m\,h_1^nh_2^n\cdots h_{m-1}^n=z^m,
\]
where the last product is $1_G$ because $h_{m-j}=h_j^{-1}$ (the cancellation is
as in~\eqref{eq:mu0}). If $K$ is harmoniously matched, we may take $z=1_G$, and
then $S$, $S^\#$ show that $G$ is harmoniously matched.
\end{proof}

\begin{remark}\label{rem:z}
Lemma~\ref{lem:split} shows that $G$ is harmonious and \#-harmonious as soon as
$z$ centralizes $H$, and that $G$ is harmoniously matched whenever moreover
$z^{|H|}=1_G$.
\end{remark}

\begin{proof}[Proof of Theorem~\ref{thm:sharp}]
Let $G_0=K$ and $G_i=G_{i-1}\ltimes H_i$. By induction and Lemma~\ref{lem:split}
(with $z=1_G$), each $G_i$ is harmoniously matched. A harmoniously matched group
is by definition harmonious and \#-harmonious.
\end{proof}

We now exhibit some non-abelian harmoniously matched groups. We use the
presentations
$D_{2n}=\langle r,s\mid r^n=s^2=(sr)^2=1\rangle$ and
$Q_{4n}=\langle a,x\mid a^{2n}=1,\ x^2=a^n,\ xax^{-1}=a^{-1}\rangle$, and
permutations are composed from right to left.

\begin{proposition}\label{prop:small}
The groups $D_{12}$, $D_{16}$, $D_{20}$, $D_{24}$, $Q_{16}$, $Q_{24}$, $A_4$ and
$S_4$ are harmoniously matched. The sequences in Table~\ref{tab:hm} satisfy
\eqref{eq:hm}.
\end{proposition}

\begin{proof}
Direct verification of the sequences in Table~\ref{tab:hm} (found and checked
by computer; the check is elementary and can also be done by hand).
\end{proof}

\begin{table}[ht]
\small
\renewcommand{\arraystretch}{1.25}
\begin{tabular}{@{}l >{\raggedright\arraybackslash}p{0.88\textwidth}@{}}
$D_{12}$ & $g$: $r^4s,\allowbreak 1,\allowbreak r^2s,\allowbreak r^5s,\allowbreak rs,\allowbreak r^2,\allowbreak r^5,\allowbreak r,\allowbreak r^4,\allowbreak r^3s,\allowbreak r^3,\allowbreak s$\\
& $g'$: $r^4s,\allowbreak r^4,\allowbreak r^5s,\allowbreak rs,\allowbreak r^5,\allowbreak r^2,\allowbreak r^3,\allowbreak r^2s,\allowbreak r,\allowbreak r^3s,\allowbreak s$\\[2pt]
$D_{16}$ & $g$: $r^4,\allowbreak r^2,\allowbreak r^3,\allowbreak r^2s,\allowbreak 1,\allowbreak r^7,\allowbreak r,\allowbreak r^7s,\allowbreak r^5s,\allowbreak rs,\allowbreak s,\allowbreak r^5,\allowbreak r^4s,\allowbreak r^6,\allowbreak r^6s,\allowbreak r^3s$\\
& $g'$: $r^4,\allowbreak r,\allowbreak r^2s,\allowbreak r^6s,\allowbreak s,\allowbreak r^7s,\allowbreak r^2,\allowbreak r^4s,\allowbreak r^5s,\allowbreak r^3,\allowbreak r^7,\allowbreak rs,\allowbreak r^5,\allowbreak r^6,\allowbreak r^3s$\\[2pt]
$D_{20}$ & $g$: $r^7,\allowbreak r,\allowbreak r^5s,\allowbreak r^3,\allowbreak r^8,\allowbreak r^9s,\allowbreak r^4s,\allowbreak s,\allowbreak r^8s,\allowbreak r^4,\allowbreak r^6,\allowbreak 1,\allowbreak rs,\allowbreak r^2s,\allowbreak r^9,\allowbreak r^6s,\allowbreak r^3s,\allowbreak r^5,\allowbreak r^2,\allowbreak r^7s$\\
& $g'$: $r^7,\allowbreak r^8s,\allowbreak r^5,\allowbreak r^6,\allowbreak s,\allowbreak r,\allowbreak r^3,\allowbreak r^5s,\allowbreak r^3s,\allowbreak r^9,\allowbreak r^8,\allowbreak r^4s,\allowbreak r^6s,\allowbreak rs,\allowbreak r^2s,\allowbreak r^9s,\allowbreak r^2,\allowbreak r^4,\allowbreak r^7s$\\[2pt]
$D_{24}$ & $g$: $r^3s,\allowbreak r^8,\allowbreak r^5,\allowbreak r^9,\allowbreak r^4s,\allowbreak s,\allowbreak r^4,\allowbreak rs,\allowbreak r^2,\allowbreak r^6,\allowbreak r^{10}s,\allowbreak r^7,\allowbreak r^3,\allowbreak r^9s,\allowbreak 1,\allowbreak r^2s,\allowbreak r^7s,\allowbreak r^8s,\allowbreak r^{10},\allowbreak r^{11},\allowbreak r,\allowbreak r^5s,\allowbreak r^{11}s,\allowbreak r^6s$\\
& $g'$: $r^3s,\allowbreak r^6,\allowbreak r^5s,\allowbreak r^3,\allowbreak s,\allowbreak r^{11}s,\allowbreak r^{11},\allowbreak r^5,\allowbreak r^2,\allowbreak r^9,\allowbreak r^7s,\allowbreak r^9s,\allowbreak r^4s,\allowbreak r^{10}s,\allowbreak r^2s,\allowbreak r^4,\allowbreak r^{10},\allowbreak r^8s,\allowbreak r^7,\allowbreak rs,\allowbreak r^8,\allowbreak r,\allowbreak r^6s$\\[2pt]
$Q_{16}$ & $g$: $a^6,\allowbreak 1,\allowbreak a^6x,\allowbreak x,\allowbreak a^4,\allowbreak a^7,\allowbreak a^2,\allowbreak ax,\allowbreak a^5x,\allowbreak a^4x,\allowbreak a^5,\allowbreak a^3x,\allowbreak a,\allowbreak a^3,\allowbreak a^2x,\allowbreak a^7x$\\
& $g'$: $a^6,\allowbreak a^2x,\allowbreak a^7,\allowbreak a^5,\allowbreak a^4,\allowbreak a^3x,\allowbreak ax,\allowbreak a^3,\allowbreak a^2,\allowbreak x,\allowbreak a^5x,\allowbreak a,\allowbreak a^4x,\allowbreak a^6x,\allowbreak a^7x$\\[2pt]
$Q_{24}$ & $g$: $a^7,\allowbreak a^6x,\allowbreak a^3,\allowbreak a^2x,\allowbreak a^8x,\allowbreak a^4,\allowbreak a^2,\allowbreak a^6,\allowbreak a^5,\allowbreak a^9,\allowbreak 1,\allowbreak a^7x,\allowbreak x,\allowbreak a^{11}x,\allowbreak ax,\allowbreak a^4x,\allowbreak a^8,\allowbreak a^3x,\allowbreak a,\allowbreak a^9x,\allowbreak a^5x,\allowbreak a^{11},\allowbreak a^{10}x,\allowbreak a^{10}$\\
& $g'$: $a^7,\allowbreak a^6x,\allowbreak a^6,\allowbreak a^9x,\allowbreak a^2,\allowbreak a^4,\allowbreak a^3,\allowbreak a^8x,\allowbreak a^3x,\allowbreak a^5,\allowbreak a^9,\allowbreak a^5x,\allowbreak a,\allowbreak a^8,\allowbreak x,\allowbreak a^2x,\allowbreak a^{10}x,\allowbreak ax,\allowbreak a^{11}x,\allowbreak a^4x,\allowbreak a^{11},\allowbreak a^7x,\allowbreak a^{10}$\\[2pt]
$A_4$ & $g$: $(123),\allowbreak \allowbreak(142),\allowbreak \allowbreak(124),\allowbreak \allowbreak(132),\allowbreak \allowbreak(243),\allowbreak \id,\allowbreak (234),\allowbreak \allowbreak(143),\allowbreak \allowbreak(14)(23),\allowbreak \allowbreak(13)(24),\allowbreak \allowbreak(12)(34),\allowbreak \allowbreak(134)$\\
& $g'$: $(123),\allowbreak \allowbreak(124),\allowbreak \allowbreak(14)(23),\allowbreak \allowbreak(142),\allowbreak \allowbreak(13)(24),\allowbreak \allowbreak(234),\allowbreak \allowbreak(143),\allowbreak \allowbreak(243),\allowbreak \allowbreak(12)(34),\allowbreak \allowbreak(132),\allowbreak \allowbreak(134)$\\[2pt]
$S_4$ & $g$: $(34),\allowbreak \allowbreak(134),\allowbreak \allowbreak(14),\allowbreak \allowbreak(14)(23),\allowbreak \allowbreak(13)(24),\allowbreak \allowbreak(12),\allowbreak \allowbreak(1423),\allowbreak \allowbreak(143),\allowbreak \allowbreak(23),\allowbreak \allowbreak(124),\allowbreak \id,\allowbreak (132),\allowbreak \allowbreak(1243),\allowbreak \allowbreak(1324),\allowbreak \allowbreak(234),\allowbreak \allowbreak(123),\allowbreak \allowbreak(24),\allowbreak \allowbreak(142),\allowbreak \allowbreak(12)(34),\allowbreak \allowbreak(243),\allowbreak \allowbreak(1432),\allowbreak \allowbreak(1234),\allowbreak \allowbreak(1342),\allowbreak \allowbreak(13)$\\
& $g'$: $(34),\allowbreak \allowbreak(124),\allowbreak \allowbreak(134),\allowbreak \allowbreak(142),\allowbreak \allowbreak(23),\allowbreak \allowbreak(24),\allowbreak \allowbreak(1234),\allowbreak \allowbreak(132),\allowbreak \allowbreak(12),\allowbreak \allowbreak(14),\allowbreak \allowbreak(12)(34),\allowbreak \allowbreak(243),\allowbreak \allowbreak(1432),\allowbreak \allowbreak(123),\allowbreak \allowbreak(1423),\allowbreak \allowbreak(14)(23),\allowbreak \allowbreak(13)(24),\allowbreak \allowbreak(234),\allowbreak \allowbreak(1243),\allowbreak \allowbreak(143),\allowbreak \allowbreak(1342),\allowbreak \allowbreak(1324),\allowbreak \allowbreak(13)$
\end{tabular}
\caption{Harmoniously matched pairs $(g,g')$ satisfying~\eqref{eq:hm}.}
\label{tab:hm}
\end{table}

\begin{example}\label{ex:Z9}
Let $G=(\Z_3\times\Z_3)\ltimes\Z_9$ be the semidirect product given by
$\varphi\colon\Z_3\times\Z_3\to\operatorname{Aut}(\Z_9)\cong\Z_9^\times$ with
$\varphi(1,0)=4$ and $\varphi(0,1)=1$. In terms of generators and relations,
\[
G=\langle a,x,y\mid a^9=x^3=y^3=1,\ xy=yx,\ xax^{-1}=a^4,\ yay^{-1}=a\rangle .
\]
Since $\Z_3\times\Z_3$ is harmoniously matched~\cite{BGHJ}, $G$ is harmoniously
matched by Lemma~\ref{lem:split}; hence $G$ is a non-abelian group of order $81$
that is both harmonious and \#-harmonious.
\end{example}

\begin{example}\label{ex:dih}
If $\gcd(m,2n)=1$, then
$D_{2nm}\cong D_{2n}\ltimes\Z_m$ and $Q_{4nm}\cong Q_{4n}\ltimes\Z_m$. Indeed, in
$D_{2nm}=\langle r,s\rangle$ the subgroup $\langle r^{n}\rangle\cong\Z_m$ is
normal with complement $\langle r^m,s\rangle\cong D_{2n}$, and in
$Q_{4nm}=\langle a,x\rangle$ the subgroup $\langle a^{2n}\rangle\cong\Z_m$ is
normal with complement $\langle a^m,x\rangle\cong Q_{4n}$ (note that
$x^2=a^{nm}=(a^m)^n$). Hence, by Theorem~\ref{thm:sharp} and
Proposition~\ref{prop:small}, $D_{16m}$ and $Q_{16m}$ ($m$ odd), $D_{12m}$,
$D_{24m}$ and $Q_{24m}$ ($\gcd(m,6)=1$), and $D_{20m}$ ($\gcd(m,10)=1$) are
harmoniously matched, hence harmonious and \#-harmonious.
\end{example}

\begin{proof}[Proof of Corollary~\ref{cor:families}]
(a) By the Schur--Zassenhaus theorem (see~\cite[Ch.~3]{Isa}), $N$ has a complement
$K\cong G/N$, so $G=K\ltimes N$ and Lemma~\ref{lem:split} applies; the abelian
case follows from~\cite{BGHJ}. (b) By Burnside's normal $p$-complement
theorem (see~\cite[Ch.~5]{Isa}), $G$ has a normal subgroup $N$ of odd order with
$G=S\ltimes N$. An abelian $2$-group is binary if and only if it is cyclic, so
$S$ is harmoniously matched by~\cite{BGHJ}, and Lemma~\ref{lem:split} applies.
(c) is Example~\ref{ex:dih}, and (d) follows from Theorem~\ref{thm:sharp} and
Proposition~\ref{prop:small}.
\end{proof}

\section{Groups of odd order}\label{sec:odd}

By Theorem~\ref{thm:main}, every group of odd order is harmonious. For the
\#-harmonious property, $\Z_3$ is an exception~\cite{BGHJ}. Lemma~\ref{lem:split}
gives the following partial result towards Conjecture~\ref{conj:odd}.

\begin{corollary}\label{cor:oddsharp}
Let $G$ be a group of odd order with a normal subgroup $H$ that has an abelian
complement $K\not\cong\Z_3$, $K\neq 1$. Then $G$ is harmoniously matched; in
particular, $G$ is \#-harmonious. For example, $\Z_p\rtimes\Z_q$ is
\#-harmonious for all primes $5\le q\mid p-1$.
\end{corollary}

\begin{proof}
An odd-order abelian group $K\notin\{1,\Z_3\}$ is harmoniously
matched~\cite{BGHJ}; apply Lemma~\ref{lem:split}.
\end{proof}

Corollary~\ref{cor:oddsharp} does not apply to $\Z_7\rtimes\Z_3$,
$\Z_{13}\rtimes\Z_3$, or to the two non-abelian groups of order $27$, since in
each of them every complement to a non-trivial proper normal subgroup is
isomorphic to $\Z_3$. Nevertheless, a computer search found \#-harmonious sequences in
all of them. For instance, with $G=\Z_7\rtimes\Z_3=\langle a,b\mid
a^7=b^3=1,\ bab^{-1}=a^2\rangle$, the sequence
\[
a^5b^2,\ a^2b^2,\ b,\ a^3b,\ a^6b,\ a^5b,\ ab,\ a^3,\ b^2,\ a,\ a^4,\ a^6,\
a^2,\ a^3b^2,\ a^2b,\ a^5,\ a^4b,\ ab^2,\ a^6b^2,\ a^4b^2
\]
is \#-harmonious. Similar searches succeeded for the non-abelian groups of orders $21$ and $27$
(the abelian groups of odd order other than $1$ and $\Z_3$ are \#-harmonious
by~\cite{BGHJ}), and for $\Z_{13}\rtimes\Z_3$,
$\Z_{19}\rtimes\Z_3$ and $\Z_7\rtimes\Z_9$. This is the evidence behind
Conjecture~\ref{conj:odd}. A natural approach would be a \#-analogue of
Lemma~\ref{lem:lift}, lifting \#-harmonious sequences through normal extensions
of odd order.

\section*{Acknowledgements}
I would like to thank the referees for their many useful comments and
suggestions.

\end{document}